\documentclass[12pt]{amsart}  
\usepackage{amssymb,amsmath,amsthm,amscd}
\usepackage[all]{xy}

\numberwithin{equation}{section}

\newtheoremstyle{fancy1}{10pt}{10pt}{\itshape}{12pt}{\textsc\bgroup}{.\egroup}{8pt}{
}
\newtheoremstyle{fancy2}{10pt}{10pt}{}{12pt}{\itshape}{.}{8pt}{ }

\theoremstyle{fancy1}

\newtheorem{thm}[equation]{Theorem}

\newtheorem*{main*}{Theorem}

\newtheorem*{cor*}{Corollary}

\newtheorem*{problem*}{Problem}

\renewcommand{\thetable}{\theequation}
\theoremstyle{fancy2}

\newtheorem*{rem*}{Remark}

\newcommand{\cref}[1]{Corollary~\ref{#1}}

\newcommand{\CP}{\mathbb{C\mkern1mu P}}
\newcommand{\HP}{\mathbb{H\mkern1mu P}}

\newcommand{\Sph}{\mathbb{S}}

\newcommand{\C}{{\mathbb{C}}}
\newcommand{\R}{{\mathbb{R}}}

\newcommand{\I}{\ensuremath{\operatorname{I}}}

\newcommand{\G}{\ensuremath{\operatorname{G}}}

\newcommand{\SO}{\ensuremath{\operatorname{SO}}}

\newcommand{\Sp}{\ensuremath{\operatorname{Sp}}}
\newcommand{\U}{\ensuremath{\operatorname{U}}}
\newcommand{\SU}{\ensuremath{\operatorname{SU}}}
\newcommand{\Spin}{\ensuremath{\operatorname{Spin}}}

\newcommand{\T}{\ensuremath{\operatorname{T}}}
\renewcommand{\S}{\ensuremath{\operatorname{S}}}

\newcommand{\fg}{{\mathfrak{g}}}
\newcommand{\fk}{{\mathfrak{k}}}

\newcommand{\fa}{{\mathfrak{a}}}

\newcommand{\fsu}{{\mathfrak{su}}}
\newcommand{\fu}{{\mathfrak{u}}}
\newcommand{\ft}{{\mathfrak{t}}}

\newcommand{\fsp}{{\mathfrak{sp}}}

\newcommand{\pro}[2]{\langle #1 , #2 \rangle}

\def\con#1=#2(#3){#1 \equiv #2 \bmod{#3}}

\newcommand{\ml}{\langle}                     
\newcommand{\mr}{\rangle}                    

\newcommand{\tr}{\ensuremath{\operatorname{tr}}}
\newcommand{\diag}{\ensuremath{\operatorname{diag}}}

\newcommand{\rk}{\ensuremath{\operatorname{rank}}}

\newcommand{\Ad}{\ensuremath{\operatorname{Ad}}}
\newcommand{\ad}{\ensuremath{\operatorname{ad}}}

\renewcommand{\sec}{\ensuremath{\operatorname{sec}}}

\newcommand{\spa}{\mbox{span}}

\newcommand{\bq}{/\!/}
\def\V{\mathcal{V}}

\newcommand{\mc}[1]{\mathcal{#1}}
\newcommand{\imag}{{\rm Im}}
\newcommand{\bb}[1]{\mathbb{#1}}
\newcommand{\HH}{\mathcal{H}}
\newcommand{\mf}[1]{\mathfrak{#1}}
\def\bsm{\begin{smallmatrix}}
\def\esm{\end{smallmatrix}}
\def\bpm{\begin{pmatrix}}
\def\epm{\end{pmatrix}}
\def\beq{\begin{equation}}
\def\eeq{\end{equation}}
\def\ms{\medskip\noindent}

\def\x{\times}

\begin{document}


\begin{titlepage}
\center{\Large{{ On
Eschenburg's Habilitation on Biquotients}}}\\[5cm]
\center{\Large{{ Lectures by Wolfgang Ziller }}}\\[1cm]
\center{\Large{{ University of Pennsylvania, 11 \& 13 Dec. 2006 }}}\\[1cm]
\center{\Large{{ Notes by Martin Kerin }}}\\[3cm]
\end{titlepage}

In 1984  Jost Eschenburg wrote his Habilitation on biquotients
\cite{E1}. This has been an important and influential paper which
has laid the foundation for the theory of biquotients in Riemannian
geometry, motivated by the Gromoll-Meyer paper from 1974 \cite{GM}
which showed that an exotic 7- sphere can be represented as a
biquotient. The new examples of positive curvature which appeared in
the Habilitation were later published in separate papers
\cite{E2},\cite{E3}. But in addition to the fact that  Eschenburg's
Habilitation is not easily accessible, the fact that it was written
in German has unfortunately prevented a general knowledge of its
full content. This motivated me to give two lectures at the
University of Pennsylvania at my secret seminar summarizing the
content of the Habilitation. I especially wanted to describe
explicitly the classification results and his tables on biquotients
of equal rank in an easily accessible form. This is potentially of
interest also outside of the subject of positive curvature. These
notes do not contain any material that cannot be found in his
Habilitation. A scanned copy of the Habilitation is available
on my home page
www.math.upenn.edu/$\sim$wziller/research.html

\section{Main Theorems}
Let $G$ be a compact Lie group with left-invariant metric $\langle
\; , \; \rangle$.  Let $K$ be the maximal subgroup of $G$ such that
$\langle \; , \; \rangle$ is right $K$-invariant.  We call $K$ the
invariance group of $\langle \; , \; \rangle$. Let $U$ be a closed
subgroup of $G \times G$, and let $U_L$, $U_R$ denote the projection
of $U$ onto the left and right factor of $G \times G$ respectively.
We assume that $U_R \subset K$, i.e. that the metric is
$U_R$-invariant.  Then $U$ acts isometrically on $G$ via
$$(u_L,u_R) \cdot g = u_L g u_R^{-1}\quad, \quad (u_L,u_R)\in U.$$
This action is free if and only if, for all $g \in G$, $g \neq e$,
we have $u_L \neq g u_R g^{-1}$. Notice also that the action is free
if and only if $u_L \neq g u_R g^{-1}$ for any $(u_L,u_R)$ in a
maximal torus in $U$. In this case $\langle \; , \; \rangle$ induces
a Riemannian metric on the quotient manifold $G \bq U$. The manifold
$G \bq U$ is called a biquotient. It is sometimes advantageous to only assume that the action is free modulo an ineffective kernel, i.e. if  $u_L = g u_R g^{-1}$ then $u_L=u_R$ lies in the center of $G$.
In the case of $SU(n)$ it is also sometimes convenient to describe the biquotient as a quotient of $SU(n)$ by $U\subset \U(n)\times\U(n)$ where the first and second component have the same determinant (although this can always be rewritten as an ordinary biquotient of $\U(n)$ as well).

 We mention that Ochiai-Takahashi
('76) showed that for any simple Lie group $G$ with left-invariant
metric $\langle \; , \; \rangle$,  the identity component of the
isometry group is contained in $ G \times G$, i.e. is of the form $
G \times K$, for some $K \subset G$. For non-simple Lie groups,
Ozeki showed that the same is true up to isometry. Thus any quotient
of a Lie group with a left invariant metric by a group of isometries
is of the above form.

We call a left invariant metric on $G$ {\it torus invariant} if it
is also right invariant under a maximal torus of $G$, i.e. $\rk K =
\rk G$.
\begin{thm}[Eschenburg]
\label{thm1} If $M^n = G \bq U$ admits a torus invariant metric with
positive curvature, then
$$\rk (G) =
\left\{ \begin{array}{cccc} \rk(U) & {\rm if } & n & \textit{ is even},\\
                           \rk(U) + 1 & {\rm if } & n & \textit{ is odd}.
\end{array} \right.$$
\end{thm}
\begin{thm}[Eschenburg]
\label{thm2} Suppose $G$ is simple and $ G \bq U$ is even
dimensional and admits a torus invariant metric with positive
curvature. Then $G \bq U$ is diffeomorphic to a homogeneous space or
$SU(3) \bq T^2$, where $T^2 = \{(diag(z,w,z w), diag(1,1,z^2 w^2))
\: | \: z,w \in S^1\}$.
\end{thm}

One easily sees  that
$$S^1_{p,q} = \{(diag(z^{p_1},z^{p_2},z^{p_3}), diag(z^{q_1},z^{q_2},z^{q_3})) \: |
\: z \in \S^1, \sum p_i = \sum q_i\}$$
acts freely on $SU(3)$ if and only if $(p_1 - q_{\sigma(1)}, p_2 -
q_{\sigma(2)}) = 1$ for all $\sigma \in \S_3$.  The resulting
biquotients $E^7_{p,q} := \SU(3) \bq \S^1_{p.q}$ are known as
Eschenburg spaces.
\begin{thm}[Eschenburg]
\label{thm3} Assume that  $G$ is semi-simple, $\rk(G) = 2$, and $
G\bq U$ is odd dimensional. If it admits a torus invariant metric
with positive curvature, then $ G\bq U$ is diffeomorphic to a
homogeneous space or $E^7_{p.q}$.
\end{thm}
Note that the hypotheses of Theorem \ref{thm3},
 together with Theorem \ref{thm1}, imply
  that $G \in \{\S^3 \times \S^3, \SU(3), \Sp(2), \G_2\}$ and
   $U \in \{\S^1, \SO(3), \SU(2)\}$.\\[.05cm]

   As explained in Eschenburg's published papers, the manifold $SU(3) \bq
   T^2$ and the biquotients $E^7_{p,q}$ with  $q_i
\not\in [\min\{p_j\}, \max\{p_j\}]$ for all $i = 1,2,3$, admit a
metric with positive
   curvature.

   \bigskip

Eschenburg's work was continued in odd dimensions in a Ph.D. Thesis
of Bock in 1995. This paper was also written in German. In addition
it was never published in any journal and we will hence shortly
describe its main result.

Consider the embedding $\Sp(2) \subset \SU(4) \subset \SU(5)$ given
by

$$
A \to \left(\begin{array}{cc|c} B & -\bar{C} &  \\ C & \bar{B} & \\
\hline & & 1\end{array}\right),
$$
where $A = B + jC$, and $B, C \in M_2 (\C)$.  Then let $\Sp(2)
\times \S^1_p \subset \U(5) \times \U(5)$ be the subgroup given by
$$\left\{
\left(\diag(z^{p_1}, \dots, z^{p_5}), \diag( A ,z)\right) \Bigg| \:
z \in \S^1, A \in \Sp(2), p = \sum p_i \right\}.$$ One easily sees that
 $\Sp(2) \cdot
\S^1_p := (\Sp(2) \times \S^1_p)/\{\pm(\I,e)\}$ acts freely on
$\SU(5)$ if and only if all of the $p_i$ are odd and $(p_{\sigma(1)}
+ p_{\sigma(2)}, p_{\sigma(3)} + p_{\sigma(4)}) = 2$ for all $\sigma
\in S_5$.  The biquotients $B^{13}_{p} := SU(5) \bq \Sp(2) \cdot
\S^1_p$ are known as Bazaikin spaces.

\begin{thm}[Bock]
\label{thm4} Suppose $M^{2n+1} = G \bq U$ admits a torus invariant
metric with positive curvature, where $U = H \cdot H'$ with $H$ of
rank one, $H'$ has no rank one factors, and $H'$ acts only
on
 one side. Then $G \bq U$ is diffeomorphic to a homogeneous space,
$E^7_{p,q}$ or $B^{13}_{p}$.
\end{thm}

The strategy in Bock's thesis is different. He does not classify the
tori $T$ which act freely on $G$ with $\rk T =\rk G-1$ since they
are too numerous. He uses the assumption that the group $H'$ is a
subgroup of $G$ with $\rk H'=\rk G -2$ and first classifies all such
pairs $(G,H')$. He then examines which further rank 1 groups can act
on both sides. He also uses along the way the 0-curvature criteria
explained in the next section to simplify the discussion.

\section{Torus-invariant metrics}

Let $T$ be a maximal torus in $G$ and $\langle \; , \; \rangle$  a
torus-invariant metric.
  Recall that each real representation of a torus
decomposes into 2-dimensional irreducible representations. In
particular, this applies to the adjoint representation of $T$ on
$\fg$.
 Therefore $\fg$ decomposes into the Lie algebra $\ft$ of $T$ and a
 sum of 2-dimensional representation modules of $T$, called {\it root
 spaces}.  The differential of the representation of $T$ on a root
 space $E$ looks as follows:\\
There is a linear form $r: \ft \to \R$ and a basis $X, Y$ for $E$,
such that for all $Z \in \ft$:

$$
[Z,X]  =  -r(Z)Y\quad , \quad [Z,Y]  =
 r(Z)X.
$$
 We denote the root space associated to each
root $r: \ft \to \R$ by $E(r)$.  Then on each $E(r)$, for all $Z \in
\ft$
$$\ad_Z = \left(\begin{array}{cc} 0 & r(Z) \\ -r(Z) & 0 \end{array}
\right),$$
i.e. $\ad_Z$ is skew-symmetric.\\
The root spaces $E(r)$ are inequivalent as $\Ad(T)$-representations.
 Therefore, with respect to any torus-invariant metric, the root
 spaces $E(r)$ and $\ft$ are pairwise orthogonal.  Hence, the
 metric is arbitrary on $\ft$, and on $E(r)$ has the form
 $\alpha_r Q|_{E(r)}$, where $Q$ is a bi-invariant metric on $G$ and
 $0 < \alpha_r \in \R$.\\[.2cm]
{\bf Example:}  $G = SU(3)$.  Then
$$\fsu(3) = \left(\begin{array}{ccc} \cdot & \Box_1 & \Box_2 \\
                                         & \cdot & \Box_3\\
                                         & & \cdot \end{array}\right),$$
where $\Box_i$, $i=1,2,3$, denote the root spaces, and $\ft = \{
\diag(a_1,a_2,a_3)\mid \sum a_i=0\}$.  Since $SU(3)$ is simple, a
bi-invariant metric is unique up to a multiple and we set
$Q(A,B)=-\frac12 \tr AB$. Thus, letting $Q$ be the Killing form,
every torus-invariant metric $\langle \; , \; \rangle$ on $SU(3)$
has the form
$$\langle
\; , \; \rangle= b + \alpha_1 Q|_{\Box_1} + \alpha_2 Q|_{\Box_2}
+ \alpha_3 Q|_{\Box_3},$$ where $b$ is an arbitrary metric on
$\ft$.

\section{Curvature of torus-invariant metrics}

The proof of Eschenburg's classification theorems consists of two
parts. In the first part one needs to classify all biquotients $G
\bq U$ with $G$ simple and $\rk(G)=\rk(U)$. In the second part one
needs to develop criteria for 0 curvature planes that can be applied
to each case. We start with the second more geometric part.

Let $\ml \; , \; \mr$ be a left invariant metric on $G$ and $Q$ a
fixed biinvariant metric. We define  the metric tensor $P\colon
\fg\to\fg$ by $\pro{X}{Y} = Q(X, P(Y))$ for $X,Y\in\fg$. We also
denote by $\ml \; , \; \mr$ the induced metric on $G \bq U$. We can
identify the vertical space $\V_g$ at $g\in G$ via left translations
with $d(L_{g^{-1}})_*(\V_g)$, which for simplicity we again denote
by $\V_g$.

\begin{thm}
The following are sufficient conditions for a zero curvature
 plane of the  metric $\ml \; , \; \mr$  on $G \bq U$ at a point
 $gU$:
\begin{enumerate}
\item[(N1)]
\label{N1} There exits a $P$-invariant abelian subalgebra $\fa$
and linearly independent $X,Y \in \fa$ which are perpendicular
to $\V_g$;

\item[(N2)]
\label{N2} There are $P$-invariant subspaces $W_1, W_2 \subset
\fg$ with $[W_1, W_2]=0$, and for some  linearly independent
 vectors $X \in W_1$, $Y \in W_2$, perpendicular to $\V_g$, we
 have
$[Y, P(Y)] \in W_2$;

\item[(N3)]
\label{N3} There is an $\Ad(K)$-invariant eigenspace $V$ of $P$
with $V \perp \fu_R$,  and for some  linearly independent
 vectors $X \in \fk$ and $Y \in V$, perpendicular to $\V_g$, we
 have
  $[P(X),Y]=0$.
\end{enumerate}
\end{thm}

\begin{proof}
We use the O'Neill formula:
$$
\sec_{G\bq U}(x, y)=\sec_G(\bar x,\bar y)+\tfrac 3 4 ||\; [X,Y]^\V\; ||^2 ,
$$
where $x,y$ are orthonormal horizontal vectors, $\bar x,\bar y$
their horizontal lift,   $X,Y$  horizontal vector fields extending
$\bar x,\bar y$, and $[X,Y]^{\V}$ denotes the vertical part of
$[X,Y]$.

For the curvature of the left invariant metric $\sec_G(\bar x,\bar
y)$ we use a formula of P\'uttmann:
\begin{align*}\label{curv}
\langle R(X,Y)Y,X)\rangle =& \tfrac 1 2
Q([PX,Y]+[X,PY],[X,Y])-\tfrac 3 4 Q(P[X,Y],[X,Y]) \\\notag
 &+ Q(B(X,Y),P^{-1}B(X,Y)) - Q(B(X,X),P^{-1}B(Y,Y))\ ,
\end{align*}
where $B (X,Y) = \tfrac{1}{2} ([X,PY] - [PX,Y])$.

For the O"Neill term $\big\| [\bar{X}, \bar{Y}]^\V \big\|^2$, one
needs to develop a formula as well. Let $a,b \in \fg$ be
left-invariant vector fields and  $A,B$ be horizontal vector fields
(with respect
  to $G \to G \bq U$) such that $A(g) = a(g)$ and $B(g) = b(g)$.  Define
   $z(a,b;g):=\|[A,B]^\V (g)\|$.  Let $\fu$ be the Lie algebra of $U$.
   The adjoint, $(\ad_a)^*$, of $\ad_a$ with respect to the left-invariant
   metric $\langle \; , \; \rangle$ is given by $(\ad_a)^* = - P^{-1} \circ \ad_a
    \circ P$.  If we define $\mc{L}(a,b) = (\ad_a)^*(b) - (\ad_b)^*(a)
     - [a,b]$, one shows that
$$z(a,b;g) = \max_{X \in \fu}
\Bigg|\frac{\pro{\Ad_{g^{-1}} X_L}{\mc{L}(a,b)} -
\pro{X_R}{[a,b]}}{|X^*(g)|} \Bigg|,$$ where $X^*$ is the action
field of $X = (X_L, X_R) \in \fu$ on $G$.

 Thus, if $a,b \in
\fg$ span a horizontal $2$-plane $\sigma_g \subset T_g G$ with
$\sec_G (a,b) = 0$ and $z(a,b; g) = 0$, then $\sigma_g$ projects to
a zero curvature plane at $g\cdot U \in G \bq U$.
 If $\pro{\:}{}$ is bi-invariant, then $\sigma_g$ has zero curvature if and only if $[a,b] = 0$.

Using all of the above, one now easily verifies the criteria
(N1)-(N3).
\end{proof}

\section{Examples}

In each of the following examples we assume that $G$ is equipped
with a torus-invariant metric, i.e. the invariance group $K$
contains a maximal
torus.\\[.3cm]
{\bf Example 1:}  There is a zero curvature plane at every point of
$\Sp(2) \bq \S^1$ for any circle $\S^1\subset \Sp(2)\x \Sp(2)$. To
see this,  consider the following subalgebras of $\fsp(2)$:
\begin{eqnarray*}
\ft &=& \{\diag(i \alpha, i \beta) \: | \: \alpha, \beta \in \R\},\\
V_1 &=& \{\diag(a_1 j + a_2 k, 0) \: | \: a_1, a_2 \in \R\},
\textrm{ and}\\
V_2 &=& \{\diag(0, b_1 j + b_2 k) \: | \: b_1, b_2 \in \R\}.
\end{eqnarray*}
$V_1$ and $V_2$ are root spaces, are hence $P$-invariant, and
satisfy $V_1 \perp V_2$ and $[V_1, V_2] = 0$.  Also, $[V_i, V_i]
\subset V_i$, $i=1,2$, since $V_1$ and $V_2$ are subalgebras.  For
all $g \in \Sp(2)$, $\V_g$ is one-dimensional and so we may find
horizontal $X \in V_1$ and $Y \in V_2$, and $X \perp Y$.  Thus by
applying (N2), we see that $\spa\{X,Y\}$ projects to a zero
curvature plane at $g \cdot \S^1 \in \Sp(2)
\bq \S^1$.\\[.3cm]
{\bf Example 2:}  Eschenburg showed that on $\! E^7_{p,q}$ there
exists a special metric which has positive curvature  if $q_i
\not\in [\min\{p_j\}, \max\{p_j\}]$ for all $i = 1,2,3$.  We will
now see that if $q_i \in [\min\{p_j\}, \max\{p_j\}]$, then
$E^7_{p,q}$ has a zero-curvature plane at some point for {\it any}
torus invariant metric. Since the other cases are similar, we assume
that $q_3 \in [\min\{p_j\}, \max\{p_j\}]$. Using the same method
that Eschenburg used to construct a metric of positive curvature in
his examples, one sees that there is a $g \in \SU(3)$ such that
$Q(\Ad_{g^{-1}} X_L - X_R, Y_3) = 0$, where
 $X = (X_L, X_R) \in \fu$, and $Y_3 = i \: \diag(1,1,-2)$.  Define $Y:= P^{-1} (Y_3)$.\\
Let
$$V_1 := \left\{\left(\begin{array}{cc|c}0 & x & 0 \\
                                        -\bar{x} & 0 & 0\\ \hline
                                        0 & 0 & 0
                      \end{array}\right) \Bigg| \:  x \in \C \right\}.$$
Then, since $V_1$ is two-dimensional, we can find $X \in V_1$ such that $X$ is horizontal at $g$.
  Since $V_1$ is a root space, we have $\Ad(K) V_1 = \Ad(T) V_1 \subset V_1$.  We also have
  $0 = [X, Y_3] = [X, P(Y)]$, and so we may apply (N3) to show that $E^7_{p,q}$ has
  a zero-curvature plane at $g \cdot \S^1_{p,q}$.\\[.3cm]
{\bf Example 3:}  We will now show that the Gromoll-Meyer sphere
$\Sigma^7 := \Sp(2) \bq \Sp(1)$ has positive curvature at a point,
but does not have positive curvature everywhere.  Here we take
$\Sp(1) = \{(\diag(q,q)\; , \;  \diag(q,1)) \: | \: q \in \Sp(1)\}
\subset \Sp(2) \times \Sp(2)$. Consider the identity $\I \in
\Sp(2)$.  We claim that $\Sigma^7$ has positive curvature at $\I
\cdot \Sp(1)$.  The vertical subspace at $\I$ is given by
$$\V_{\I} = \left\{ \left(\begin{array}{cc}0 & 0 \\ 0 &
x \end{array}\right) \big| \: x \in \imag (\bb{H}) \right\}
 \subset \fsp(2).$$
Hence the horizontal subspace at $\I$ is given by
$$\HH_{\I} = \left\{ \left(\begin{array}{cc} y & v \\ -\bar{v} & 0
\end{array}\right) \big| \: y \in \imag (\bb{H}), v \in \bb{H}
\right\} \subset \fsp(2),$$ which coincides with the horizontal
space $\mf{H}_{\I}$ at $\I$ of
 the homogeneous space $\Sph^7 = \Sp(2)/\Sp(1)$, where in this case
 $\Sp(1) = \{\diag(1,q) \: | \: q \in \Sp(1) \} \subset \Sp(2)$.
 Now, since $\Sph^7$ has positive curvature, any vectors
 $X,Y \in \mf{H}_{\I}$ such that $[X,Y]=0$ must therefore be
 linearly dependent.  Hence there are no horizontal zero-curvature
 planes at $\I$, and so $\Sigma^7$, with the metric induced by a biinvariant metric,
 has non negative curvature and has positive curvature at
 $\I \cdot \Sp(1)$.\\
We now show that for any $U_R$-invariant metric on $\Sp(2)$,
$\Sigma^7$ has a plane of zero-curvature at some point, where $U_R =
\{\diag(q,1) \: | \: q \in Sp(1)\}$.  Consider the subspaces $W_1,
W_2, W_3 \subset \fsp(2)$, where
\begin{eqnarray*}
W_1 &:=& \left\{ \left(\begin{array}{cc} x & 0 \\ 0 &
0 \end{array}\right) \bigg| \: x \in \imag (\bb{H}) \right\},\\
W_2 &:=& \left\{ \left(\begin{array}{cc} 0 & 0 \\
0 & y \end{array}\right) \bigg| \: y \in \imag (\bb{H})
 \right\}, \;\; \textrm{and}\\
W_3 &:=& \left\{ \left(\begin{array}{cc} 0 & v \\
 -\bar{v} & 0 \end{array}\right) \bigg| \: v \in \bb{H} \right\}.
\end{eqnarray*}
$\Ad(U_R)$ acts on $W_1$, $W_2$ and $W_3$ by $\{x \to qxq^{-1}\}$,
$\{{\rm id}\}$ and $\{v \to qv\}$ respectively.
 Then $W_1$, $W_2$ and $W_3$ are clearly inequivalent
 $\Ad(U_R)$-representations, and by Schur's Lemma, are pairwise
 orthogonal.  We remark that the metric on $W_2$ can be arbitrary,
 since $\Ad(U_R)$ acts trivially.\\
The vertical subspace at $g \in \Sp(2)$ is given by
$$\V_g = \left\{ \Ad_{g^{-1}} \left(\begin{array}{cc} x & 0 \\
0 & x \end{array}\right) -  \left(\begin{array}{cc} x & 0 \\ 0 & 0
\end{array}\right) \bigg| \: x \in \imag(\bb{H}) \right\}.$$
Let $g = \frac{1}{2}\left(\begin{array}{cc} 1 & i \\ i & 1
\end{array}\right)$.  Then
$$\V_g = span\left\{ \left(\begin{array}{cc} 0 & 0 \\ 0 & i \end{array}\right),
\left(\begin{array}{cc} j & k \\ k & 0 \end{array}\right),
\left(\begin{array}{cc} -k & j \\ j & 0 \end{array}\right)
\right\}.$$ Let $X = \left(\begin{array}{cc} i & 0 \\ 0 & 0
\end{array}\right) \in W_1$ and $Y = \left(\begin{array}{cc} 0 & 0
\\ 0 & b \end{array}\right) \in W_2$, where $b \perp i$.  Therefore
$X,Y \in \HH_g$.  Applying (N2) we see that
$\sigma = \spa\{X,Y\}$ projects to a zero-curvature plane at $g \cdot \Sp(1) \in \Sigma^7$.\\[.3cm]
{\bf Example 4:}  For any $U_R$-invariant metric on $\SO(2n+1)$, the
biquotient $M := \Delta \SO(2) \backslash \SO(2n+1) / \SO(2n-1)$ has
a zero-curvature plane at every point.  Here
$$\Delta SO(2) = \left\{\left(\begin{array}{cccc} A & & & \\
                                                    & \ddots & & \\
                                                    & & A & \\
                                                    & & & 1 \end{array}\right) \Bigg| \: A \in
                                                    \SO(2)\right\} \subset \SO(2n+1).$$
Note that $\rk(\SO(2n+1)) = \rk( \Delta \SO(2) \times \SO(2n-1))$,
and also that $M$ is, in fact, the quotient of the unit tangent
bundle, $T_1 \Sph^n$, of $\Sph^n$,
and the action of $\Delta SO(2)$ on $T_1 \Sph^n$ is the geodesic flow.\\
Let $V,W = \R^{2n-1}$.  Then we may write
$$\mf{so}(2n+1) = \left\{\left(\begin{array}{c|cc}  &  &  \\
                                                \mf{so}(2n-1) & x & y\\
                                                     &   &  \\
                                   \hline  -x^t & 0 & a \\
                                           -y^t & -a & 0
                               \end{array}\right)
\Bigg| \: x \in V, y \in W, a \in \R \right\}.$$
Then $\mf{so}(2n-1)^\perp = V \oplus W \oplus \R$.  Now
$\Ad(SO(2n-1))$ acts on $\mf{so}(2n-1)^\perp$ via
$$\Ad_A (x,y,a) = (Ax, Ay, a).$$
Since $V \perp W$, we may choose $g$ such that $\Ad_g V \perp \Ad_g
W$.  Now, all $\Ad(\SO(2n-1))$-invariant subspaces are of the form
$\Ad_g V$, for some $g$ of the form
$$\left(\begin{array}{c|cc} \I & & \\
                           \hline  & a & b\\
                           & c & d
                           \end{array}\right) \in \SO(2).$$
Noting that $[V,W] = 0$, we now choose $X \in \Ad_g V$, $Y \in \Ad_g W$ such that
 $[X,Y] = 0$.  Note also that $P(\Ad_g V) \subset \Ad_g V$ and $P(\Ad_g W) \subset \Ad_g W$.\\
Then for all $h \in \SO(2n+1)$ we use the fact that $\Delta \SO(2)$
is one-dimensional
 to choose $X,Y$ as above such that $X,Y \perp \V_h := \Ad_{h^{-1}}
 (\Delta \mf{so}(2)) - \mf{so}(2n-1)$.

We may now apply (N1) to see that there is thus a zero-curvature
plane at every point of $\Delta \SO(2) \backslash \SO(2n+1) /
\SO(2n-1)$.

 Note that Wilking ('02) equipped $\Delta \SO(2)
\backslash \SO(2n+1) / \SO(2n-1)$ with a metric of almost positive
curvature, i.e. a metric which has positive curvature on an open
dense set of points.

\bigskip

\section{Classification of biquotient  actions by a maximal torus}

\bigskip

A major achievement in Eschenburg's Habilitation is the
classification of all biquotients $G \bq U$ where $G$  is simple and
$\rk(G) = \rk(U)$. This is based on first classifying all maximal
tori which act freely on $G$, which we will describe now.

\bigskip

We start with free biquotient actions on $\SU(n)$.

\begin{thm}
\label{SU(n)tori}
    For $n \geq 3$ a torus $\T^{n-1} = \ml z,w_1, \dots, w_{n-2}\mr$
    acts freely on $SU(n)$ if and only if it either acts on one side, or
     $\T^{n-1}$ is
    conjugate to  $\S_{1,\ell}$ or $\S_{2,\ell}$ for some
    $1 \leq \ell \leq \frac{n}{2}$, where
$$
\S_{1,\ell}^{n-1}=\big\{\diag((z^2,\dots , z^2,1,\dots ,1) \, ;\,
\diag((z \bar w_1 \ldots \bar w_{n-2}) , z^2 w_1,
\ldots,z^2 w_{\ell - 1}, w_\ell ,\ldots, w_{n-2},z )\big\}
$$
$$
\S_{2,\ell}^{n-1}=\big\{\diag(1,\dots ,1,z^2) \, ;\,
\diag((z\bar w_1 \dots \bar w_{\ell-1}) ,  w_1,
\ldots,w_{n-2}, (z \bar w_\ell \dots \bar w_{n-2}) )\big\}
$$
The actions of $\S_{1,\ell}$ and $\S_{2,\ell}$ are equivalent if and
only if $\ell = 1$.
\end{thm}

\ms {\it Remark 1.}  Recall that we must have $\det(u_L) =
\det(u_R)$, where $(u_L, u_R) \in U = \S_{i,\ell}$, $i=1,2$.  Thus
in $\S_{1,\ell}$ there are $\ell$ copies of $z^2$ on the left-hand
side.

\ms {\it Remark 2.}  If $n = 2m$, we may rewrite  the actions in
Theorem \ref{SU(n)tori} with $\ell=m$ as

$$
\S_{1,m}^{2m-1}=\big\{\diag((z,\dots , z,\bar z,\dots ,\bar z) \, ;\,
\diag((\bar w_1 \dots \bar w_{n-2}) ,  w_1,
\dots, w_{n - 2},1 )\big\}
$$
$$
\S_{2,m}^{2m-1}=\big\{\diag(z,\dots ,z,z^{n-1}) \, ;\,
\diag((\bar w_1 \dots \bar w_{m-2}) ,  w_1,
\ldots,w_{n-2}, (\bar w_m \dots \bar w_{n-2}) )\big\}
$$

\ms {\it Remark 3.}    Note that in the biquotient actions  on
$\SU(n)$ there is only one $\S^1$ which acts on
both sides of $\SU(n)$.\\[.5cm]

\begin{thm}
\label{Sp(n)tori} For $n \geq 3$ an $n$-torus $\T^n = \ml z, w_1,
\dots, w_{n-1} \mr$ acts freely on  $\Sp(n)$  if and only if it
either acts on one side, or
     $\T^{n}$ is
    conjugate to  $P_1^n$ or $P_2^n$, where

$$
P_1^n=\big\{\diag((1,\dots , 1, z) \, ;\,
\diag(w_1,
\dots, w_{n - 1},(\bar w_1 \dots \bar w_{n-1}) )\big\}
$$
$$
P_2^n=\big\{\diag(z,\dots ,z) \, ;\,
\diag( w_1,
\ldots,w_{n-1}, 1 )\big\}
$$
The actions of $P_1^n$ and $P_2^n$ are equivalent if and only if $n
= 2$.
\end{thm}

\bigskip

If we consider  the usual embeddings
            $$U(n) \subset \SO(2n)  \subset \SO(2n+1)
            \quad ,\quad U(n) \subset \Sp(n)$$
for $n \geq 2$, we see that a maximal torus in $\U(n)$ can  also be
viewed as a maximal torus in $\SO(2n), \SO(2n+1)$ or $\Sp(n)$. Thus
free biquotient actions on $\Sp(n)$ give rise to free biquotient
actions on $\SO(2n)$ and $ \SO(2n+1)$, and vice versa. We do not consider
the group $\SO(4)$ since it is not simple.

\smallskip

We record the following special cases:

\begin{cor*} For the rank 2 groups we have:
\begin{itemize}
\item[(a)] The only $2$-torus acting freely on $\SU(3)$ on both sides is:
$$
\S_{1,1}^2=\S_{2,1}^2=\big\{\diag(1,1,z^2w^2) \, ;\,
\diag(z,w,zw)\big\}
$$
\item[(b)] The only $2$-torus acting freely on $\Sp(2)$ on both sides is:
$$
P_1^2=P_2^2=\big\{\diag(z,z) \, ;\,
\diag( w,1 )\big\}
$$
\end{itemize}
\end{cor*}

\bigskip

\ms {\it Remark.} Due to the isomorphism $\Spin(6) = \SU(4)$, we
have, besides the free actions of $P_1^3,P_2^3$ on $\SO(6)$, a third
torus acting freely which can be written as:
$$
P_3^3=\big\{\diag(z,z,z) \, ;\,
\diag( zw_1,w_2,\bar w_1 \bar w_2)\big\}
$$

\bigskip

Finally, for the exceptional Lie groups, we have:

\begin{thm}
\label{exceptori}
        The exceptional Lie groups $G_2, F_4, E_6, E_7, E_8$
        admit no free two-sided torus actions of maximal rank.
\end{thm}

\bigskip

\section{Classification of maximal biquotient actions of maximal rank}

We now describe the classification when $U$ is not abelian. Note
that if a maximal torus $T$ acts freely on $G$, then any extension
$U$ of $T$
 with $\rk(U) = \dim T$ will also act freely on $G$.  Eschenburg
  classified all such $U$ which are maximal among these extensions,
  i.e. all $U_{max} \subset G \times G$ such that if $U$ is another
  extension of $T$, then $U$ is contained in some $U_{max}$.  Given
  such a $U_{max}$ and a torus $T \subset U_{max}$, maximal in
  $U_{max}$, it is an simple exercise to list all the extensions
  $U$ with $T \subset U \subset U_{max}$ by using the Borel Siebenthal classification of
  maximal rank subgroups. For example,
if $\T^{n-1}$ is the usual maximal torus in $\SU(n)$, then
extensions $U$ are given by
$$\T^{n-1} \subset \S(\U(n_1) \times \dots \times \U(n_k)) \subset \SU(n),$$
where $\sum n_i = n$. Similarly, for $\T^n$ the usual maximal torus
in $\Sp(n)$, extensions $U$ are given by
$$\T^n \subset U_1 \times \dots \times U_k \subset \Sp(n),$$
where $U_i = \Sp(n_i)$ or $\U(n_i)$, and $\sum n_i = n$.

Finally, if $\T^n$ is the usual maximal torus in $\SO(2n)$ or
$\SO(2n+1)$, extensions $U$ are given by
$$\T^n \subset U_1 \times \dots \times U_k \subset \SO(n) ,$$
where $U_i = \SO(n_i)$ or $\U(n_i/2)$, and $\sum n_i = n$.

\bigskip

We break up the description of all such maximal $U$ into those where
the quotient is diffeomorphic to a rank one symmetric space in Table
\ref{cross}, and those which are not in Table \ref{notcross}.

\renewcommand{\thetable}{\Alph{table}}

\renewcommand{\arraystretch}{1.4}
\begin{table}[!hbtp]
      \begin{center}
          \begin{tabular}{|c||c|c||c|c|c|}
\hline
& $G$ & $n$ & $T_U$ & $U=U_1\x U_2$ & $G \bq U$ \\
\hline
\hline

1 & $\SU(n)$     & $n \geq 5$      & $\S_{1,\ell}, \ \ 2 \leq \ell <
\frac{n}{2}$     &
$\S^1_\ell \ltimes \SU(n-1)$        & $\CP^{n-1}$ \\
\hline

2 & $\SU(2n)$    & $n \geq 2$      & $\S_{1,n}$       & $\Delta \SU(2) \times \SU(2n-1)$      & $\HP^{n-1}$ \\
\hline \hline

3 & $\Spin(7)$     & & $P_1^3$      & $ \Spin(3) \times \G_2$       & $\Sph^4$ \\
\hline

4 & $\Spin(8)$     & & $P_1^4$      & $ \Spin(3) \times \Spin(7)'$       & $\Sph^4$ \\
\hline

5 & $\Spin(9)$     & & $P_1^4$      & $ \Spin(3) \times \Spin(7)'$       & $\HP^3$ \\
\hline

6 & $\SO(2n)$      & $n \geq 3$    & $P_2^n$       & $\Delta \SO(2) \times \SO(2n-1)$      & $\CP^{n-1}$\\
\hline

7 & $\SO(4n)$      & &  $P_2^{2n}$       & $\Delta \SU(2) \times \SO(4n-1)$      & $\HP^{n-1}$\\
\hline \hline

8 & $\Sp(n)$       & $n \geq 2$    & $P_2^n$     & $\Delta \Sp(1) \times \Sp(n-1)$        & $\HP^{n-1}$ \\
\hline

          \end{tabular}
      \end{center}
      \vspace{0.1cm}
      \caption{Maximal rank free actions such that $G \bq U$ is diffeomorphic to a compact rank one symmetric space.}\label{cross}
\end{table}

\begin{table}[!hbtp]
      \begin{center}
          \begin{tabular}{|c||c|c||c|c|}
\hline
& $G$ & $n$ & $\T_U$ & $U=U_1\x U_2$ \\
\hline
\hline

9 & $\SU(n)$     & $n \geq 5$      & $\S_{2,\ell}, \ \ 2 \leq \ell < \frac{n}{2}$
     & $\S^1 \ltimes \SU(\ell)\SU(n-\ell)$  \\
\hline

10 & $\SU(2n)$    & $n \geq 2$      & $S_{2,n}$       & $\S^1 \times \SU(n)\SU(n)$  \\
\hline \hline

11 & $\SO(2n)$     & $n \geq 5$     & $P_1^n$      & $\SO(3) \times \SU(n)$      \\
\hline

12 & $\SO(2n+1)$     & $n \geq 5$     & $P_1^n$      & $\SO(3) \times \SU(n)$    \\
\hline

13 & $\SO(2n+1)$     & $n \geq 3$     & $P_2^n$      & $\Delta \SO(2) \times \SO(2n-1)$  \\
\hline

14 & $\SO(2n)$      & $\left.\begin{array}{c} 2n = p + q \geq 2,\\
p, q \ {\rm odd}\end{array} \right.$         & $P_2^n$
                    & $\Delta \SO(2) \times \SO(p)\SO(q)$\\
\hline

15 & $\SO(4n+1)$      & $n \geq 2$    & $P_2^{2n}$       & $\Delta \SU(2) \times \SO(4n-1)$    \\
\hline \hline

16 & $\Sp(n)$        & $n \geq 3$    & $P_1^n$     & $\Sp(1) \times \SU(n)$    \\
\hline

\hline

17 & $\Sp(4)$       & & $P_1^4$      & $\Sp(1) \times \SU(2)^3$\\
\hline
          \end{tabular}
      \end{center}
      \vspace{0.1cm}
      \caption{Maximal rank free actions such that $G \bq U$ is not diffeomorphic
      to a compact rank one symmetric space.}\label{notcross}
\end{table}

In these Tables, $U=U_1\x U_2$ where $U_1$ is a rank one factor
which, except in case 1 and 9, is embedded only on the left. In all cases $U_2$ acts only on the right.

For case 1, $U$ is a semidirect product $\S^1 \ltimes \SU(n-1)$
where $\SU(n-1)=\{\diag(A,1)\mid A\in \SU(n-1)\}$ acts only on the
right and
$$
\S^1=\big\{\diag(z^2,\dots ,z^2,1, \dots ,1) \, ;\,
\diag( z,z^2\dots,z^2,1,\dots,1,z)\big\}
$$
on both sides. Indeed, this circle subgroup $\S^1\subset G\x G$
clearly normalizes $(\{e\}\x \diag(A,1))$. Similarly for case 9,
where
$$
\S^1=\big\{\diag(1,\dots ,1, z^2) \, ;\,
\diag( z,1\dots,1,z)\big\}.
$$
In case 10 on the other hand, the circle acts only on the left as
$\diag(z,\dots,z,z^{n-1})$.

The diagonal subgroup: $\Delta\SO(2)\subset\SO(2n)$ and
$\Delta\SU(2)\subset\SO(4n)$ act as Hopf actions on
$\Sph^{2n-1}=\SO(2n)/\SO(2n-1)$  respectively
$\Sph^{4n-1}=\SO(4n)/\SO(4n-1)$. Furthermore,
$\Delta\SO(2)\subset\SO(2n+1)$ and $\Delta\SU(2)\subset\SO(4n+1)$
are obtained via first embedding into $\SO(2n)$ respectively
$\SO(4n)$.

In case 11,12, 16,17 the embedding of $U_1$ is the standard block
embedding.

The embedding of $U_2$, which only acts on the right, is the
standard block embedding when $U_2$ is a classical group. In case 3
it is the standard embedding of $\G_2$ in $\SO(7)$, lifted to
$\Spin(7)$ and in case 4 and 5 the spin embedding.

The only other embedding that needs to be described is the embedding
of $U_2=\SU(2)^3\subset\Sp(4)$ in case 17. For this we consider the
representation $\SU(2)^3\subset\SU(8)$ given by the exterior tensor
product of the tautological 2 dimensional representation of $\SU(2)$
on each factor. Since this representation is symplectic, the image
lies in $\Sp(4)$.

\bigskip

We  point out that entry 14 was missing in its full generality in
\cite{E1}, as was observed in \cite{EKS}. Due to this fact it is not
yet certain that his classification is complete. This should be settled in the forthcoming English translation by
Catherine Searle and Jost Eschenburg.

\bigskip

We finally mention his classification in the rank 2 case of all
biquotients, not just the equal rank ones:

\begin{thm}
If $G$ is a simple Lie group of rank 2 and $U$ a rank one group
acting freely as a biquotient, then it is either a homogeneous space, or $U$ is the circle action
on $\,\SU(3)$ whose quotient is the Eschenburg spaces $E^7_{pq}$, or
the Gromoll-Meyer biquotient action of $\,\Sp(1)$ on $\Sp(2)$ (or
its subgroup $\S^1\subset\Sp(1)\;$), or a biquotient action of
$\,\SU(2)$ (respectively $\S^1\subset\SU(2)\;$) on $\G_2$. The
latter acts via the index 3 three dimensional subgroup on the left,
and the index 4 three dimensional one on the right.
\end{thm}

The biquotient $\Sp(2)/\!/\Sp(1)$ is the famous Gromoll Meyer sphere \cite{GM}, which is homeomorphic but not diffeomorphic to a sphere.
In \cite{KZ} it was shown that $\G_2/\!/\SU(2)$ is homeomorphic to $T_1\Sph^6$, but it is not known if it diffeomorphic to it or not. In the case of $\Sp(2)/\!/\S^1$ and $\G_2/\!/\S^1$ we do not know which 2-sphere bundle over the respective $\SU(2)$ biquotient it is.

In \cite{KS},\cite{K2} a diffeomorphism classification was given of (almost all) Eschenburg spaces $E_{p,q}$ in terms of number theoretic sums. This was used in \cite{CEZ} to study various homeomorphism and diffeomorphism properties of these manifolds. See also \cite{AMP1,AMP2,E4,K1,Sh} for other topological properties.

\providecommand{\bysame}{\leavevmode\hbox
to3em{\hrulefill}\thinspace}

\end{document}